\documentclass[reqno, 11pt, a4paper]{amsart}

\usepackage{
    a4wide,
    amssymb,
    bbm,
    centernot,
    mathtools,
} 
\usepackage{graphicx}
\graphicspath{{./images/}}
\usepackage[cal=euler, scr=boondoxo]{mathalfa}
\usepackage[colorlinks=true, linkcolor=blue,citecolor=blue]{hyperref}
\usepackage{float}

\usepackage{graphicx}
\usepackage{amssymb}
\usepackage{amsthm}
\usepackage{tikz}

\usepackage{centernot}
\usepackage{mathtools}
\usepackage{ stmaryrd }

\usepackage{hyperref}
\hypersetup{allcolors=black,
pdftitle={NV Scheme to SLE},
    pdfpagemode=FullScreen}
\numberwithin{equation}{section}
\usepackage{bbm}

\usepackage{braket}
\usepackage{amsmath,amsthm,amssymb}
\usepackage{physics}
\usepackage{mathtools}
\usepackage{bm}
\usepackage{esvect}
\usepackage[english]{babel}
\usepackage{extarrows}

\usepackage{centernot}
\usepackage{mathtools}
\usepackage{ stmaryrd }
\usepackage{svg}
\usepackage{hyperref}
\hypersetup{allcolors=black,
pdftitle={Multiple SLE and Perturbations},
    pdfpagemode=FullScreen}

\usepackage{setspace}
\usepackage{amsmath}

\numberwithin{equation}{section}
\usepackage{bbm}

\usepackage{braket}
\usepackage{amsmath,amsthm,amssymb}
\usepackage{physics}
\usepackage{mathtools}
\usepackage{bm}
\usepackage{esvect}
\usepackage[english]{babel}

\newtheorem{theorem}{Theorem}[section]
\newtheorem{lemma}[theorem]{Lemma}
\newtheorem{proposition}[theorem]{Proposition}

\newtheorem{corollary}[theorem]{Corollary}

\theoremstyle{definition}
\newtheorem{definition}[theorem]{Definition}

\theoremstyle{remark}
\newtheorem{remark}[theorem]{Remark}

\numberwithin{equation}{section}

\begin{document}

\definecolor{airforceblue}{RGB}{204, 0, 102}
\newenvironment{draft}
  {\par\medskip
  \color{airforceblue}%
  \medskip}

\title[Perturbations of multiple Schramm-Loewner evolution]{Perturbations of multiple Schramm-Loewner evolution with two non-colliding Dyson Brownian motions}


\author{Jiaming Chen}
\address{Departement Mathematik, ETH Zürich}
\curraddr{101, Rämistrasse, CH-8092 Zürich, Switzerland}
\email{jiamchen@student.ethz.ch}
\thanks{}

\author{Vlad Margarint}
\address{NYU-ECNU Institute of Mathematical Sciences, NYU Shanghai}
\curraddr{1555 Century Avenue, Shanghai, China}
\email{margarint@nyu.edu}
\thanks{}

\begin{abstract}
In this article we study multiple $SLE_\kappa$, for $\kappa\in(0,4]$, driven by Dyson Brownian motions. This model was introduced in the unit disk by Cardy \cite{cardy2003} in connection with the Calogero-Sutherland model. We prove the Carath\'eodory convergence of perturbed Loewner chains under different initial conditions and under different diffusivity $\kappa \in (0,4]$ for the case of $N=2$ driving forces. Our proofs use the analysis of Bessel processes and estimates on Loewner differential equation with multiple driving forces. In the last section, we estimate the Hausdorff distance of the hulls under perturbations of the driving forces, with assumptions on the modulus of the derivative of the multiple Loewner maps.
\end{abstract}

\dedicatory{}
\maketitle

\tableofcontents

\section{Introduction}
\label{sec:intro}
The forward multiple Loewner chain encodes the dynamics of a family of conformal maps $g_t(z)$ defined on simply connected domains $\mathbb{H}\backslash K_t$ of the upper-half plane $\mathbb{H}$, where $K_t$ are growing hulls (\cite{Antti}, $Sec.\;4.1.2$) in the sense that $K_s\subseteq K_t$ for all $0\leq s\leq t$. In this work we study a Loewner chain generated by $N\in\mathbb{N}$ real continuous driving forces $\{\lambda_1(t),\lambda_2(t),\ldots,\lambda_N(t)\}$.  Denote by $\lambda_j:\mathbb{R}\to\mathbb{R}$ these driving functions, we have 
\begin{equation}
\label{1.1}
    \partial_tg_t(z)=\frac{1}{N}\sum\limits_{j=1}^N\frac{2}{g_t(z)-\lambda_j(t)}
\end{equation}
with $g_0(z)=z$.
This work is motivated by \cite{cardy2003}, where the author establishes a connection between the Calogero-Sutherland model and the multiple $SLE_\kappa$ driven by Dyson Brownian motions. This model was studied in more details in \cite{VHealey}. In this paper, we treat further the multiple $SLE_\kappa$ driven by Dyson Brownian motions. In order to define this object, we consider the Weyl chamber (\cite{Grabiner}, $Sec.\;4.$) given by
\begin{equation*}
    \mathcal{M}_N\coloneqq\big\{\mathbf{x}\in\mathbb{R}^N;\;x_1<x_2<\cdots<x_N\big\}.
\end{equation*}
Throughout the paper we work with $\big(\Omega, \mathcal{F},\mathbb{P}\big)$, the standard probability space. Let $B_j(t)$, $j=1,\ldots,N$ be one-dimensional standard independent Brownian motions defined on this space. The Dyson Brownian motions with diffusivity parameter $\kappa \in (0,4]$ are defined by a set of differential equations in the following
\begin{equation}
\label{1.3}
    d\lambda_j(t)=\frac{1}{\sqrt{2}}dB_j(t)+\frac{2}{\kappa}\sum\limits_{1\leq k\leq N,j\neq k}\frac{dt}{\lambda_j(t)-\lambda_k(t)}
\end{equation}
with $(\lambda_1(0),\ldots\lambda_N(0))\in\mathcal{M}_N$, for all $t\in\mathbb{R}_+$ and $j=1,\ldots,N$. We refer the reader to \cite{Dys} for more details.
In the literature \cite{Katorihydro} on Dyson Brownian motions, the parameter of the system is usually denoted by $\beta=\frac{8}{\kappa}>0$, where $\kappa$ is the diffusivity parameter in $(1.3)$. The Dyson Brownian motions have a unique strong solution (\cite{YaueErdodynamical}, $Thm.\;12.2$) that we use as the simultaneous driving force for our multiple $SLE_\kappa$.
An intuitive picture is that $\{\lambda_1(t),\lambda_2(t),\ldots,\lambda_N(t)\}$ describes an ensemble of diffusing particles (\cite{MARTIN}, $Rmk.\;2.4$) in which they repel each other via a Coulomb force. It is known that when $\kappa\in(0,4]$, no two Dyson Brownian particles will collide (\textit{i.e.} touch $\partial\mathcal{M}_N$) almost surely. Actually, these particles do not collide for all $\kappa\in(0,8]$. To be precise, denote by
\begin{equation}
\label{1.4}
    \tau_N\coloneqq\inf\{0\leq t\leq T;\;\exists~i,j\; s.t.\; \abs{\lambda_i(t)-\lambda_j(t)}=0\}.
\end{equation}
Then $\tau_N=\infty$ almost surely as in (\cite{kkk}, $Prop.\;3.1$). This result justifies our choice of an arbitrary compact time interval $[0,T]$. Meanwhile, it is known (\cite{Makoto}, $Thm.\;1.3$) that when $\kappa\in (0,4]$, the transformations $g_t(z)$ map a simply connected domain $\mathbb{H}\backslash K_t$ conformally onto the upper-half plane $\mathbb{H}$, where $K_t$ consists of the image of $N$ non-intersecting simple curves. In other words, the map $g_t: \mathbb{H}\setminus \cup_{j=1}^n\gamma_t^j \to \mathbb{H}$ is a conformal isomorphism. Each curve corresponds to a driving force $\lambda_j(t)$, $ j \in \mathbb{N}$. We focus on this case and throughout this article we assume $\kappa \in (0,4]$.\par

 In the last years, there have been many results on the multiple $SLE$ model (see \cite{katori2019}, \cite{Vik}, \cite{peltola2019global}, \cite{HotaScheis}, \cite{delMonacoSchleis}, \cite{Katoriyoshida}, \cite{dubedat}, \cite{BeffPeltoWu}, \cite{Karrila1}, \cite{MonHottaScheiss}, \cite{Katorihydro}, \cite{Kytolla}, \cite{Zhan2curves}, \cite{Zhan2curves2}, \cite{Roth}) for a non-exhaustive list of papers where the model is studied in the upper half-plane, unit disk, either in the simultaneous growth case, or in the non-simultaneous growth case. 
 

In this paper, we mainly focus on how the forward Loewner chain $g_t(z)$ behaves when the system is under different perturbations. We propose an estimate of such perturbations in the sense of Carathéodory convergence in the following sections. We are interested in what happens when either the initial value of the driving forces $\lambda_j(t)'s$ or their diffusivity parameter $\kappa$ are perturbed in the $N=2$ driving forces case. In our work, we study the random driving forces case with constant weights. We use probabilistic properties of the driving forces whereas the deterministic intermediate results are performed for general $N$ case, compared with \cite{Roth} where the analysis is detailed in the $N=2$ case with non-constant weights (under assumptions) using Gr\"onwall's inequality.


The paper contains three parts. In Section $2$, we introduce the Carathéodory convergence of multiple Loewner chains and obtain some preliminary estimates. In Section $3$, we analyze two types of perturbations of the multiple $SLE_\kappa$. The first type of perturbation is on the initial value of driving forces; the second type of perturbation is on the diffusivity parameter $\kappa \in (0.4]$. The first situation has an  application in Statistical Physics and refers to a natural question when people consider the stability of statistical physics models under boundary perturbations, such as multiple Ising model interfaces when $\kappa=3$.
The second situation is also very natural as seen from the existing $SLE_\kappa$ literature on the topic. For example, this was studied in the sup-norm in a sequence of papers \cite{Vikcont}, \cite{YuanFriz}, \cite{me1} in the one-curve case for the continuity in the parameter $\kappa \leq 4 $ of the welding homeomorphism.
In Section $4$, we perform a variant estimate of the Hausdorff distance between the perturbed hulls under the backward multiple Loewner chain.

The study of the multiple $SLE_\kappa$ with Dyson Brownian driving forces has many avenues to be explored in the following years. Motivated by the Random Matrix Theory (RMT) for the cases $\beta=1$ (Gaussian Orthogonal Ensemble), $\beta=2$ (Gaussian Unitary Ensemble), and beyond for general values of $\beta>0$, there are many tools currently available in the study of the Dyson Brownian motions.


Additionally, there are links with CFT (Conformal Field Theory) as explained in the original paper of Cardy \cite{cardy2003} where the model has been studied. For more works on the connection
between multiple SLE and CFT we refer to \cite{Vik} and \cite{Kazu}.

The case $N=2$ Dyson particles is a first step in this direction in which one can analyse the corresponding model using the analysis of the Bessel process which has already been well studied in the literature. The distance between the two particles and the statistics of the extremal driving force can serve as a model for the analysis of the general $N$ and asymptotic $N\to\infty$ cases in which the statistics of distances are being studied in the context of RMT and beyond, in particular, in the case of Dyson Brownian motion and Airy process which models the extremal eigenvalue for general $\beta$ (see \cite{Zeitouni}). Moreover, for general $N$ case, the joint density of the Dyson particles is known (see \cite{Zeitouni}), which then motivates a study of the situation where the Dyson Brownian motion starts from the equilibrium. In the $N=2$ case, the results can be further refined with a thorough analysis of the Bessel process, for which there exists ample literature on known results. Through adaptation and following a similar method of the proof to Lemma 6.15 in \cite{Antti} for $N=2$ curves, we can obtain probabilistic properties of the curves via estimates on the statistics of the driving forces in this setting.

In general, given the structure at the level of driving forces (for $N=2$ and for general $N$), one can introduce new observables such as the statistics of the $k^{th}$ largest distance between driving forces or the probability to have no driving forces in a symmetric region about the origin (see \cite{noeigen} for $\beta=2$), etc. And such new observables are adopted in order to use them in the study of the scaling limits of discrete models (see \cite{Nitzschner}, \cite{DJian}). The variety of new observables follow naturally from the structure of the driving forces. In the $N=2$ case, the analysis simplifies due to Bessel process techniques that  are available in the existing literature on stochastic analysis. This can serve as a good platform for more general cases. The analysis in this paper can be thought as a first-step towards the general $N$ case, and asymptotic $N\to\infty$, where there are techniques that involve the study of local statistical properties such as the study of the gaps between particles, and other local statistics developed in RMT (see \cite{BenArous}, \cite{Schlein}, \cite{Landon} for a non-exhaustive list).

\section{Carathéodory convergence of Loewner chains}
\label{sec:chapter 1}
Throughout this paper, we use $\norm{\cdot}_{[0,T]}$ for the uniform norm on the interval $[0,T]$, and denote by $\norm{\cdot}_{[0,T]\times G}$ the uniform norm on the product space $[0,T]\times G$, where $G\subseteq\mathbb{H}$ is compact. Also, throughout the paper we consider the coupling of the forward Loewner chains in the sense that they are driven by Dyson Brownian motions.\par
In this section, we propose an estimate to the perturbation of forward Loewner chain $g_t(z)$ in the sense of Carathéodory convergence. The central idea is the uniform convergence on compact sets. This type of convergence is emphasized in complex analysis (\cite{rudin}, $Thm.\;10.28$), where when a sequence of holomorphic functions Carathéodory converges to a limit function, then taking the limit preserves the holomorphicity.\par


One can study the Carath\'eodory convergence in the setting when $G$ is an arbitrary compact subset of $\mathbb{H}$. Indeed, if we want a Carathéodory estimate where compact sets $G$ could run freely over the domain $\mathbb{H}\backslash K_t$, $t\in[0,T]$, then the results must be discussed pathwise.
\begin{definition}
    Denote by $D\subseteq\mathbb{H}$ a simply connected domain. Let $f_n(t,z):[0,T]\times D\to\mathbb{H}$ be a sequence of conformal maps, and let $f(t,z):[0,T]\times D\to\mathbb{H}$ be a conformal map. We say $f_n$ converges in the Carathéodory sense to $f$, or $f_n\xRightarrow{Cara}f$, if for each compact $G\subseteq D$, the sequence $(f_n)_{n\in\mathbb{N}}$ converges to $f$ uniformly on $[0,T]\times G$.
\end{definition}
The estimate on the $g_t(z):\mathbb{H}\backslash K_t\to\mathbb{H}$ corresponding to different perturbations is based on Definition $2.1$. For notational convenience, we will write $g_t(z)$ as $g(t,z)$ from now on. Consider two forward Loewner chains $g_1(t,z)$ and $g_2(t,z)$ defined on $[0,T]\times\mathbb{H}\backslash K_t$, where $K_T$ is the union of the hulls of $g_1(t,z)$ and $g_2(t,z)$ up to time $T$, that is $K_T=\cup_{k=1}^2K_T^k$. Suppose they are generated $resp.$ by $N$ continuous driving forces $\{V_{k,1}(t),\ldots,V_{k,N}(t)\}$ with $k=1,2$. Then we have the following estimate.

\begin{figure}[htbp]
\centering
\includegraphics[scale=0.34]{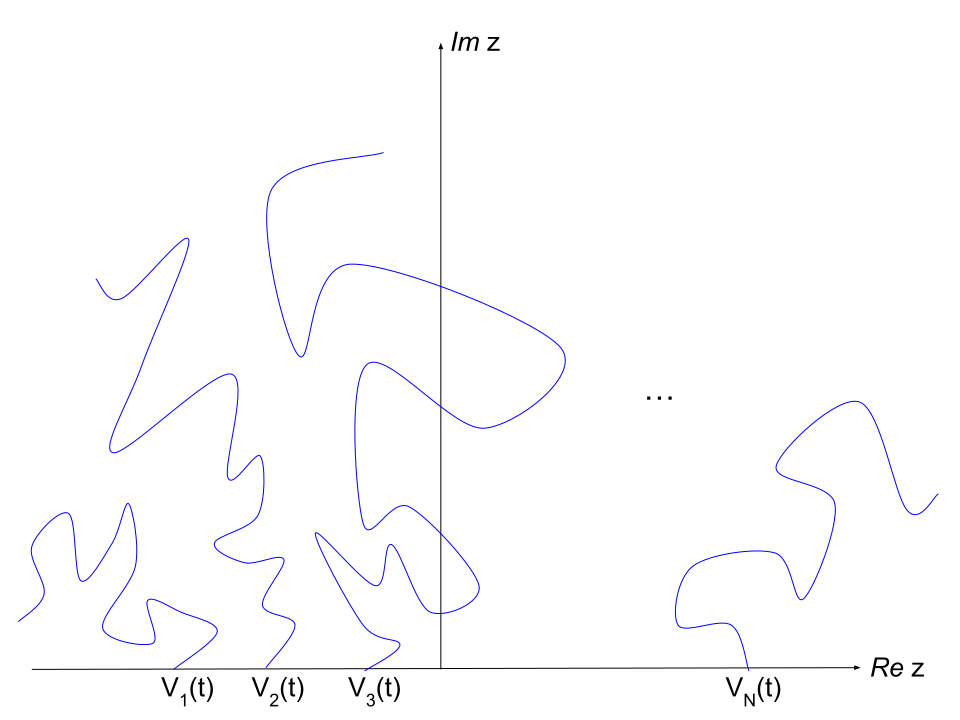}
\caption{Schematic picture of simultaneously growing N-tuples of the non-intersecting slits in $\mathbb{H}$ at time $t\in[0,T]$. See another nice illustration in \cite{Katorihydro}.}
\end{figure}

In the above illustration we have dropped the $k$-indices and omitted the perturbation, for convenience of the reader.

\begin{proposition}
    For an arbitrary compact $G\subseteq\mathbb{H}\backslash K_T$, for all $z \in G$, there exists a constant $C(T,G)>0$ such that we have
    \begin{equation*}
        \norm{g_1(t,z)-g_2(t,z)}_{[0,T]\times G}\leq C(T,G)\sum\limits_{j=1}^N\norm{V_{1,j}(t)-V_{2,j}(t)}_{[0,T]}.
    \end{equation*}
\end{proposition}
\begin{proof}
    It is by $(1.1)$ that we have the constraint
    \begin{equation*}
        \partial_tg_k(t,z)=\frac{1}{N}\sum\limits_{j=1}^N\frac{2}{g_k(t,z)-V_{k,j}(t)}
    \end{equation*}
    with $k=1,2$. Choose arbitrarily $z_1,z_2\in G$. Let $\psi(t)\coloneqq g_1(t,z_1)-g_2(t,z_2)$. And we have
    \begin{equation}\begin{aligned}
    \label{2.3}
        \frac{d}{dt}\psi(t)&=\partial_tg_1(t,z_1)-\partial_tg_2(t,z_2)\\
        &=\frac{1}{N}\sum\limits_{j=1}^N\bigg(\frac{2}{g_1(t,z_1)-V_{1,j}(t)}-\frac{2}{g_2(t,z_2)-V_{2,j}(t)}\bigg)\\
        &=\frac{1}{N}\sum\limits_{j=1}^N\xi_j(t)\bigg(g_1(t,z_1)-V_{1,j}(t)-g_2(t,z_2)+V_{2,j}(t)\bigg),
    \end{aligned}\end{equation}
    where we define
    \begin{equation*}
        \xi_j(t)\coloneqq\frac{-2}{\big(g_1(t,z_1)-V_{1,j}(t)\big)\cdot\big(g_2(t,z_2)-V_{2,j}(t)\big)}
    \end{equation*}
    for $j=1,\ldots,N$. Additionally we define $D_j(t)\coloneqq V_{1,j}(t)-V_{2,j}(t)$ for each $j$. Combined with $(2.3)$, then we have
    \begin{equation*}
        \frac{d}{dt}\psi(t)=\frac{1}{N}\sum\limits_{j=1}^N\xi_j(t)\big(\psi(t)-D_j(t)\big).
    \end{equation*}
    At this moment, we observe that
    \begin{equation*}
        \frac{d}{dt}\bigg(e^{-\frac{1}{N}\sum\limits_{j=1}^N\int_0^t\xi_j(s)ds}\cdot\psi(t)\bigg)=-\frac{1}{N}\sum\limits_{j=1}^N\xi_j(t)D_j(t)\cdot e^{-\frac{1}{N}\sum\limits_{j=1}^N\int_0^t\xi_j(s)ds}
    \end{equation*}
    and consequently
    \begin{equation*}
        \psi(t)=e^{\frac{1}{N}\sum\limits_{j=1}^N\int_0^t\xi_j(s)ds}\cdot\psi(0)-\frac{1}{N}\sum\limits_{j=1}^N\int_0^tdu\cdot\xi_j(u)D_j(u)\cdot e^{\frac{1}{N}\sum\limits_{j=1}^N\int_u^t\xi_j(s)ds}.
    \end{equation*}
    On the other hand, we have the following inequality
    \begin{equation*}
        \abs{e^{\frac{1}{N}\sum\limits_{j=1}^N\int_0^t\xi_j(s)ds}}\leq e^{\frac{1}{N}\sum\limits_{j=1}^N\int_0^t\abs{\xi_j(s)}ds}.
    \end{equation*}
    Then, we know that
    \begin{equation*}\begin{aligned}
        &\abs{\frac{1}{N}\sum\limits_{j=1}^N\int_0^tdu\cdot\xi_j(u)D_j(u)\cdot e^{\frac{1}{N}\sum\limits_{j=1}^N\int_u^t\xi_j(s)ds}}\leq \frac{1}{N}\sum\limits_{j=1}^N\norm{D_j(t)}_{[0,T]}\cdot\int_0^tdu\cdot\abs{\xi_j(u)}e^{\frac{1}{N}\sum\limits_{j=1}^N\int_0^u\abs{\xi_j(s)}ds}\\
        &\;\;\;\;\;\;\;\;\;\;\;\;\;\;\;\;\;\;\;\;\;\;\;\leq\bigg(\sum\limits_{j=1}^N\norm{D_j(t)}_{[0,T]}\bigg)\cdot\int_0^tdu\cdot\frac{1}{N}\sum\limits_{j=1}^N\abs{\xi_j(u)}\cdot e^{\frac{1}{N}\sum\limits_{j=1}^N\int_0^u\abs{\xi_j(s)}ds}\\
        &\;\;\;\;\;\;\;\;\;\;\;\;\;\;\;\;\;\;\;\;\;\;\;=\bigg(\sum\limits_{j=1}^N\norm{D_j(t)}_{[0,T]}\bigg)\cdot\bigg(e^{\frac{1}{N}\sum\limits_{j=1}^N\int_0^t\abs{\xi_j(s)}ds}-1\bigg).
    \end{aligned}\end{equation*}
    Moreover, by the Cauchy-Schwartz inequality, we have
    \begin{equation*}
        \frac{1}{N}\sum\limits_{j=1}^N\int_0^t\abs{\xi_j(s)}ds\leq\frac{1}{N}\sum\limits_{j=1}^N\sqrt{I_{1,j}\cdot I_{2,j}},
    \end{equation*}
    where we define $I_{k,j}$ for $k=1,2$ and $j=1,\ldots,N$ in the following
    \begin{equation*}
        I_{k,j}\coloneqq\int_0^t\frac{2}{\abs{g_k(s,z_k)-V_{k,j}(s)}^2}ds \leq N\cdot\log\frac{\Im z_k}{\Im g_k(t,z_k)},
    \end{equation*}
    because 
    \begin{equation*}
        \partial_t\Im g_k(t,z_k)=\frac{2}{N}\sum_{j=1}^N\frac{-\Im g_k(t,z_k)}{\abs{g_k(t,z_k)-V_{k,j}(t)}^2}
    \end{equation*}
    similar to the one-curve case in (\cite{Antti}, $Sec.\;6.2.2$). In fact, with the compact $G\subseteq\mathbb{H}\backslash K_T$, there exists $\delta_1(G)>0$ such that $\Im g_k(T,z)\geq\delta_1(G)$ for all $z\in G$, $k=1,2$. Hence, we have
    \begin{equation*}
        I_{k,j}\leq N\cdot\log\frac{\Im z_k}{\max\bigg\{\delta_1(G),\sqrt{\big((\Im z_k)^2-4t\big)^+}\bigg\}}
    \end{equation*}
    where $x^+=\max\{x,0\}$. Since $t\in[0,T]$ and $z_1,z_2\in G$ where $G$ is compact in $\mathbb{H}\backslash K_T$, we could choose $\delta_2(G)\coloneqq\text{dist}(G,\mathbb{R})>0$ and define a constant by
    \begin{equation*}
        C(T,G)^{\frac{1}{N}}\coloneqq\frac{\delta_2(G)}{\max\bigg\{\delta_1(G),\sqrt{\big(\delta_2(G)^2-4t\big)^+}\bigg\}}.
    \end{equation*}
    Here we have $I_{k,j}\leq\log C(T,G)$ for all $k$ and $j$. Hence, we know that
    \begin{equation*}
        \abs{\psi(t)}\leq e^{\frac{1}{N}\sum\limits_{j=1}^N\log C(T,G)}\cdot\abs{\psi(0)}+\bigg(\sum\limits_{j=1}^N\norm{D_j(t)}_{[0,T]}\bigg)\cdot\big(C(T,G)-1\big).
    \end{equation*}
    Therefore, we conclude
    \begin{equation*}
        \abs{g_1(t,z_1)-g_2(t,z_2)}\leq C(T,G)\cdot\big(\sum\limits_{j=1}^N\norm{V_{1,j}(t)-V_{2,j}(t)}_{[0,T]}+\abs{z_1-z_2}\big)
    \end{equation*}
    for all $t\in[0,T]$ and $z_1,z_2\in G$. Now choose $z_1=z_2=z$ and take supremum over the left side, we arrive at our final result
    \begin{equation*}
        \norm{g_1(t,z)-g_2(t,z)}_{[0,T]\times G}\leq C(T,G)\cdot\sum\limits_{j=1}^N\norm{V_{1,j}(t)-V_{2,j}(t)}_{[0,T]}.
    \end{equation*}
\end{proof}
\begin{remark}
    With slight changes, the above argument can be adapted to the backward multiple Loewner maps, and similar Carathéodory estimates can be obtained.
\end{remark}

\section{Perturbations}
\label{sec:chapter 2}
We are particular interested in the forward Loewner map driven by Dyson Brownian motions. In this section, we restrict our attention to the $N=2$ case. The general $N$-curve case is no different but only adds extra complexity. When $N=2$, we have two driving forces $\{\lambda_1(t),\lambda_2(t)\}$ that are interacting Dyson Brownian motions. Their evolution is described in the following equation
\begin{equation}\begin{aligned}
\label{3.1}
      d\lambda_1(t)&=\frac{2}{\kappa}\cdot\frac{dt}{\lambda_1(t)-\lambda_2(t)}+\frac{1}{\sqrt{2}}dB_1(t)\\
    d\lambda_2(t)&=\frac{2}{\kappa}\cdot\frac{dt}{\lambda_2(t)-\lambda_1(t)}+\frac{1}{\sqrt{2}}dB_2(t)
\end{aligned}\end{equation}
with $\lambda_1(0)=a_1,\;\lambda_2(0)=a_2,\;a_1>a_2$, where $B_1(t)$ and $B_2(t)$ are independent one-dimensional Brownian motions. Here we consider only $\kappa\in(0,4]$. In this case, if we think of $\lambda_k(t)$, $k=1,2$ as indication of two particles, then these two particles never collide in $[0,T]$ for all $T\in\mathbb{R}_+$. In other words, the stopping time defined in $(1.4)$ satisfies $\tau_2=\infty$ almost surely as in (\cite{Henri}, $Prop.\;1.$) because we consider Bessel processes of dimension $d=1+\frac{8}{\kappa}\geq3$, given $\kappa\leq 4$ .\par
Let $X_t\coloneqq\lambda_1(t)-\lambda_2(t)$. Based on the above observations, we know $X_t>0$ for all $t\in[0,T]$ almost surely. We further observe that
\begin{equation}
\label{3.2}
    dX_t=\frac{4}{\kappa}\cdot\frac{dt}{X_t}+dW_t
\end{equation}
with $X_0=a_1-a_2$, and $W_t\coloneqq\frac{1}{\sqrt{2}}\big(B_1(t)-B_2(t)\big)$ is a Wiener process. Choose $d=1+\frac{8}{\kappa}$, then $X_t$ admits the canonical form of $d$-dimensional Bessel process with the constraint
\begin{equation*}
    dX_t=\frac{d-1}{2}\cdot\frac{dt}{X_t}+dW_t.
\end{equation*}
In this section we discuss two types of perturbations. The first type of perturbation is varying the initial value of driving forces. The second type is varying the diffusivity parameter $\kappa \in (0,4]$. The study in both cases involves the analysis of transient Bessel processes in dimension $d\geq3$.
\subsection{Perturbation of the initial value}
The first type of perturbation is to slightly change the initial value of $\lambda_k(0)$ for $k=1,2$. With the initial value perturbed, we get a different set of Dyson Brownian motions. Our goal is to estimate the difference of the forward Loewner chains driven by these varying forces.\par
To be precise, choose $0<\epsilon<\frac{1}{3}(a_1-a_2)$ and select $b_k$ in the $\epsilon$-ball of $a_k$ for $k=1,2$ to be the perturbed initial value of the Dyson Brownian motions. Then $b_1>b_2$ and we arrive at another set of perturbed Dyson Brownian motions $\{\eta_1(t),\eta_2(t)\}$ given by
\begin{equation*}\begin{aligned}
    d\eta_1(t)&=\frac{2}{\kappa}\cdot\frac{dt}{\eta_1(t)-\eta_2(t)}+\frac{1}{\sqrt{2}}dB_1(t)\\
    d\eta_2(t)&=\frac{2}{\kappa}\cdot\frac{dt}{\eta_2(t)-\eta_1(t)}+\frac{1}{\sqrt{2}}dB_2(t)
\end{aligned}\end{equation*}
with $\eta_1(0)=b_1$ and $\eta_2(0)=b_2$. Notice that the process $\eta_k(t)$ still contains the same Brownian motion $B_k(t)$, $k=1,2$, because the perturbation affects only the initial value.\par
In this two-force case, we denote by $g_\lambda(t,z)$ the original forward Loewner chain generated by forces $\{\lambda_1(t),\lambda_2(t)\}$ and by $g_\eta(t,z)$ the perturbed forward Loewner chain generated by forces $\{\eta_1(t),\eta_2(t)\}$. Hence, we have
\begin{equation*}\begin{aligned}
    \partial_tg_\lambda(t,z)&=\frac{1}{g_\lambda(t,z)-\lambda_1(t)}+\frac{1}{g_\lambda(t,z)-\lambda_2(t)}\\
    \partial_tg_\eta(t,z)&=\frac{1}{g_\eta(t,z)-\eta_1(t)}+\frac{1}{g_\eta(t,z)-\eta_2(t)}
\end{aligned}\end{equation*}
with $g_{\lambda}(0,z)=z$ and $g_{\eta}(0,z)=z$ for all $z\in\mathbb{H}$. We continue using $X_t=\lambda_1(t)-\lambda_2(t)$ to denote the gap between two interacting Brownian forces $\lambda_k(t),\;k=1,2$. As shown in $(3.2)$, $X_t$ is a Bessel process with dimension $1+\frac{8}{k}$ and initial value $X_0=a_1-a_2$. Denote by $Y_t\coloneqq\eta_1(t)-\eta_2(t)$ the gap between $\eta_k(t),\;k=1,2$. Then $Y_t$ is a Bessel process with the same dimension $d=1+\frac{8}{\kappa}$ and satisfies
\begin{equation*}
    dY_t=\frac{4}{\kappa}\cdot\frac{dt}{Y_t}+dW_t
\end{equation*}
with $Y_0=b_1-b_2$. Observe the Bessel processes $X_t$ and $Y_t$ are driven by the same Wiener process $W_t$. Hence their difference $X_t-Y_t$ satisfies
\begin{equation}
\label{3.7}
    d(X_t-Y_t)=-\frac{4}{\kappa}\cdot\frac{X_t-Y_t}{X_tY_t}dt.
\end{equation}
Denote by $a\coloneqq a_1-a_2$ and $b\coloneqq b_1-b_2$. Integrate both sides on $(3.7)$ and we observe that
\begin{equation*}
    X_t-Y_t=(a-b)\cdot e^{-\frac{4}{\kappa}\int_0^t\frac{1}{X_sY_s}ds}.
\end{equation*}
Notice that the term $\frac{1}{X_tY_t}$ evolves stochastically. For $k=1,2$, we could observe that
\begin{equation*}\begin{aligned}
    d\lambda_k(t)-d\eta_k(t)&=\frac{2}{\kappa}\bigg(\frac{1}{\lambda_k(t)-\lambda_{3-k}(t)}-\frac{1}{\eta_k(t)-\eta_{3-k}(t)}\bigg)dt\\
    &=(-1)^k\frac{2}{\kappa}\cdot\frac{X_t-Y_t}{\big(\lambda_k(t)-\lambda_{3-k}(t)\big)\cdot\big(\eta_k(t)-\eta_{3-k}(t)\big)}dt.
\end{aligned}\end{equation*}
Hence, we have
\begin{equation*}
    d\lambda_k(t)-d\eta_k(t)=(-1)^k(a-b)\frac{2}{\kappa}\cdot e^{-\frac{4}{\kappa}\int_0^t\frac{1}{X_sY_s}ds}\cdot\frac{1}{X_tY_t}dt
\end{equation*}
for $k=1,2$. The above equation admits an integral form
\begin{equation}\begin{aligned}
\label{3.11}
    \lambda_k(t)-\eta_k(t)&=a_k-b_k+(-1)^k(a-b)\frac{2}{\kappa}\int_0^te^{-\frac{4}{\kappa}\int_0^s\frac{1}{X_uY_u}du}\cdot\frac{1}{X_sY_s}ds\\
    &=a_k-b_k+\frac{1}{2}(-1)^{3-k}(a-b)\bigg(e^{-\frac{4}{\kappa}\int_0^t\frac{1}{X_sY_s}ds}-1\bigg).
\end{aligned}\end{equation}
At this point, we have an explicit form to $\lambda_k(t)-\eta_k(t)$. Looking back to Proposition $2.2$, naturally we want to have an estimate to $g_\lambda(t,z)-g_\eta(t,z)$ in the Carathéodory sense.
\begin{proposition}
    For all $0<\epsilon<\frac{a}{3}$, let $H_T=\mathbb{H}\backslash K_T$, where $K_T=\cup_{k=1}^2K_T^k.$ Choose $b_k\in\mathbb{R}$ with $\abs{a_k-b_k}<\epsilon$ for $k=1,2$.  Let $g_\lambda(z)$ and $g_\eta(z)$ be two multiple Loewner chains induced by Dyson Brownian motions $\{\lambda_1(t),\lambda_2(t)\}$ and $\{\eta_1(t),\eta_2(t)\}$, resp. Suppose $\lambda_k(0)=a_k$ and $\eta_k(0)=b_k$ for $k=1,2$. Then almost surely we have 
    \begin{equation*}
        \norm{g_\lambda(t,z)-g_\eta(t,z)}_{[0,T]\times G}<4C(T,G)\cdot\epsilon\;\;\text{for all}\;G\;\text{compact in}\;H_T.
    \end{equation*}
\end{proposition}
\begin{proof}
    At this moment, we already know $X_t,Y_t>0$ for all $t\in[0,T]$ almost surely using properties of Bessel processes of dimension $d \geq 3$. Inspect $(3.11)$, we know for $k=1,2$ that
    \begin{equation*}
        \abs{\lambda_k(t)-\eta_k(t)}\leq\abs{a_k-b_k}+\frac{1}{2}\abs{a-b}\cdot\bigg(1-e^{-\frac{4}{\kappa}\int_0^T\frac{1}{X_tY_t}dt}\bigg)<2\epsilon.
    \end{equation*}
    By Proposition $2.2$, we know that
    \begin{equation*}\begin{aligned}
        \norm{g_\lambda(t,z)-g_\eta(t,z)}_{[0,T]\times G}&\leq C(T,G)\cdot\sum\limits_{k=1}^2\norm{\lambda_k(t)-\eta_k(t)}_{[0,T]}\\
        &<4C(T,G)\cdot\epsilon.
    \end{aligned}\end{equation*}
    And the proposition is verified.
\end{proof}
\begin{remark}
    In Proposition $3.1$, $H_T$ and $K_T$ are random sets in $\mathbb{H}$, which depend pathwisely on the forward Loewner chain.
\end{remark}
So far we have estimated $g_\lambda(t,z)-g_\eta(t,z)$ in the Carathéodory sense under a perturbation of initial value of driving forces. In practice, when computing a forward multiple Loewner chain driven by Dyson Brownian motions, we could approximate its initial value and the convergence is guaranteed in the Carathéodory sense. Indeed, we have the following result.
\begin{corollary}
    Suppose $g_t(z):\mathbb{H}\backslash K_T\to\mathbb{H}$ is a forward Loewner chain induced by Dyson Brownian motions $\{\lambda_1(t),\lambda_2(t)\}$ with initial value $\lambda_1(0)>\lambda_2(0)$. Additionally, suppose there is a sequence of forward Loewner chains $g^n_t(z):\mathbb{H} \backslash K_T^n\to \mathbb{H}$, with hulls $K_T^n\subseteq\mathbb{H}$, induced by Dyson Brownian motions $\{\lambda^n_1(t),\lambda^n_2(t)\}$ with $\lambda^n_1(0)>\lambda^n_2(0)$ and approaching initial value $\lambda^n_k(0)\xrightarrow{n}\lambda_k(0)$. Then we know that the sequence of conformal maps $\{g_t^n\}_{n\in\mathbb{N}}$ converges in the Carathéodory sense to the conformal map $g_t$ in the non-trivial open domain $\mathbb{H}\backslash (K_T\cup_1^\infty K_T^n)$.
\end{corollary}


\subsection{Perturbation of the diffusivity parameter}
The second type of perturbation is with respect to the diffusivity parameter $\kappa \in (0,4]$. This type of perturbation has been considered extensively in the one-curve case as mentioned in the introduction. We recall that in the current work we have always chosen $\kappa\in(0,4]$ so that there is no phase transition (\cite{Vlad}, $Sec.\;3.$) corresponding to the $(1+\frac{8}{\kappa})$-dimensional Bessel process. When there is perturbation, $\kappa$ is varied and we have a new diffusivity parameter $\kappa^*\in(0,4]$ such that $\kappa^*\neq\kappa$. The difference in parameter results in different Dyson Brownian motions, and therefore different forward Loewner chains.\par
To simplify the model, we assume $\kappa^*>\kappa$ without loss of generality. Denote by $\{\lambda_1(t),\lambda_2(t)\}$ the original Dyson Brownian motions. Their dynamics is described in $(3.1)$ with initial value $\lambda_k(0)=a_k,\,k=1,2,\,a_1>a_2$. Denote by $\{\lambda_1^*(t),\lambda_2^*(t)\}$ the perturbed Dyson Brownian motions. They respect the following equations
\begin{equation*}\begin{aligned}
    d\lambda_1^*(t)&=\frac{2}{\kappa^*}\cdot\frac{dt}{\lambda_1^*(t)-\lambda_2^*(t)}+\frac{1}{\sqrt{2}}dB_1(t)\\
    d\lambda_2^*(t)&=\frac{2}{\kappa^*}\cdot\frac{dt}{\lambda_2^*(t)-\lambda_1^*(t)}+\frac{1}{\sqrt{2}}dB_2(t)\\
\end{aligned}\end{equation*}
with initial value $\lambda_k^*(0)=\lambda_k(0)=a_k,\;k=1,2$. Let $K_T=\cup_{k=1}^2K^k_T,$ with $j=1$ corresponding to the parameter $\kappa \in (0, 4]$ and $j=2$ corresponding to the parameter $\kappa^* \in (0,4].$ We have $g(t,z):[0,T]\times\mathbb{H}\backslash K_T\to\mathbb{H}$ the original Loewner chain generated by forces $\{\lambda_1(t),\lambda_2(t)\}$. And we denote by $g^*(t,z):[0,T]\times\mathbb{H}\backslash K_T(\omega)\to\mathbb{H}$ the perturbed Loewner chain generated by $\{\lambda_1^*(t),\lambda_2^*(t)\}$.

The evolution respects
\begin{equation}
\label{3.16}
    \partial_tg^*(t,z)=\frac{1}{g^*(t,z)-\lambda^*_1(t)}+\frac{1}{g^*(t,z)-\lambda^*_2(t)}
\end{equation}
with $g^*(0,z)=z$ for all $z\in\mathbb{H}\backslash K_T$. Denote by $X_t$ the gap between $\lambda_1(t)$ and $\lambda_2(t)$. Then $X_t$ is a $(1+\frac{8}{\kappa})$-dimensional Bessel process with initial value $X_0=a$. Its evolution is described in $(3.2)$. Similarly, let $X^*_t\coloneqq\lambda_1^*(t)-\lambda_2^*(t)$ be the gap of the two perturbed driving forces. The $X^*_t$ respects the following equation
\begin{equation}
\label{3.17}
    dX^*_t=\frac{4}{\kappa^*}\cdot\frac{dt}{X^*_t}+dW_t
\end{equation}
with $X_0^*=X_0=a$ and where $W_t$ is the Wiener process defined in $(3.2)$. Notice here $X^*_t$ is a $(1+\frac{8}{\kappa^*})$-dimensional Bessel process.\par
Our main goal is to give an probabilistic estimate of $g(t,z)-g^*(t,z)$ in the Carathéodory sense. Following Proposition $2.2$, we need first estimate the sup-norm of $\lambda_k(t)-\lambda_k^*(t)$ for $k=1,2$. Indeed, define the indices of the Bessel processes by $\nu\coloneqq\frac{4}{\kappa}-\frac{1}{2}$, $\nu^*\coloneqq\frac{4}{\kappa^*}-\frac{1}{2}$. Before proving our result, we have the following lemma. Elements of this lemma were kindly provided by H. Elad Altman in a personal communication.
\begin{lemma}
    Given a $(1+\frac{8}{\kappa^*})$-dimensional Bessel process $X_t^*$ and a $(1+\frac{8}{\kappa})$-dimensional Bessel process $X_t$ with $4 \geq \kappa^*>\kappa>0$ and the same initial value $X^*_0=X_0=a>0$, we have almost surely that
    \begin{equation*}
        \sup\limits_{0\leq s\leq t}(X^*_s-X_s)^2\leq\frac{4t}{\kappa^2}(\kappa^*-\kappa).
    \end{equation*}
\end{lemma}
\begin{proof}
    Observe $(3.2)$ and $(3.17)$, we see
    \begin{equation*}
        X^*_t-X_t=\frac{4}{\kappa^*}\int_0^t\frac{ds}{X^*_t}-\frac{4}{\kappa}\int_0^t\frac{ds}{X_t}.
    \end{equation*}
    Using Itô's lemma, we have
    \begin{equation}\begin{aligned}
    \label{3.20}
        d(X^*_t-X_t)^2&=2(X^*_t-X_t)\cdot\bigg(\frac{4}{\kappa^*X^*_t}-\frac{4}{\kappa X_t}\bigg)dt\\
        &=\frac{4}{\kappa^*\kappa}(\kappa-\kappa^*)\cdot\frac{X^*_t-X_t}{X^*_t}dt+\frac{8}{\kappa}(X^*_t-X_t)\cdot(\frac{1}{X^*_t}-\frac{1}{X_t})dt.
    \end{aligned}\end{equation}
    At the same time, we have that $(X^*_t-X_t)\cdot(\frac{1}{X^*_t}-\frac{1}{X_t})\leq0$ for all $t\in[0,T]$ almost surely. Integrating both sides and we obtain
    \begin{equation*}
        (X^*_t-X_t)^2\leq(\kappa-\kappa^*)\frac{4}{\kappa^*\kappa}\cdot\int_0^t\frac{(X^*_s-X_s)_+}{X^*_s}ds.
    \end{equation*}
    On the other hand, we have that $(X^*_s-X_s)_+\leq X^*_s$. By considering $\kappa^*\leq\kappa$, we have the conclusion
    \begin{equation*}
        \sup\limits_{0\leq s\leq t}(X^*_s-X_s)^2\leq\frac{4t}{\kappa^2}(\kappa-\kappa^*).
    \end{equation*}
\end{proof}
We denote by $S_t=\sup\limits_{0\leq s\leq t}W_t$ the supremum Brownian motion. And we are ready to state the main result.
\begin{theorem} 
    Let $g(t,z)$ and $g^*(t,z)$ be two multiple Loewner chains for the parameters $\kappa,\kappa^*\in(0,4]$, resp. Choose arbitrary compacts $G\subseteq H_T=\mathbb{H}\backslash K_T,$ where $K_T=\cup_{k=1}^2K_T^k$, with $K_T^1$ corresponding to the parameter $\kappa \in (0, 4]$ and $K_T^2$ to the parameter $\kappa^* \in (0, 4]$. Then there exist $\alpha_1,\alpha_2,\alpha_3>0$ depending only on $(T,G,a,\kappa)$ such that if we further define
    \begin{equation*}\begin{aligned}
        \varphi(x)&\coloneqq\alpha_1x^{1/8}+\alpha_2x^{1/4}+\alpha_3x^{7/8}\\
        \zeta(x)&\coloneqq2\frac{x^{\nu/8}}{a^{2\nu}}+2x^{3/4}e^{-1\big/2Tx^{3/2}}
    \end{aligned}\end{equation*}
    for all $x\in\mathbb{R}_+$, then $\lim\limits_{x\to0^+}\varphi(x)=0$, $\lim\limits_{x\to0^+}\zeta(x)=0$ almost surely and we have
    \begin{equation*}
        \mathbb{P}\bigg(\norm{g(t,z)-g^*(t,z)}_{[0,T]\times G}>\varphi(\kappa^*-\kappa),\;\forall~G\;\text{compact in}\;H_T\bigg)<\zeta(\kappa^*-\kappa).
    \end{equation*}
    In particular, the probability that the conformal map $g^*(t,z)$ deviates from $g(t,z)$ for at least $\varphi(\kappa^*-\kappa)$ in the Carathéodory topology is less $\zeta(\kappa^*-\kappa)$.
\end{theorem}
\begin{proof}
    From $(3.1)$ and $(3.16)$, we see for $k=1,2$ that
    \begin{equation*}\begin{aligned}
        d\lambda_k(t)-d\lambda^*_k(t)&=\frac{2}{\kappa}\cdot\frac{dt}{\lambda_k(t)-\lambda_{3-k}(t)}-\frac{2}{\kappa^*}\cdot\frac{dt}{\lambda^*_k(t)-\lambda^*_{3-k}(t)}\\
        &=(-1)^{3-k}\frac{2}{\kappa^*\kappa}\cdot\frac{\kappa^*X^*_t-\kappa X_t}{X^*_tX_t}dt.
    \end{aligned}\end{equation*}
    To obtain an expression of $\lambda_k(t)-\lambda^*_k(t)$, we need to express the process $\kappa^*X^*_t-\kappa X_t$. Indeed, we have
    \begin{equation}
    \label{3.26}
        \kappa^*dX^*_t-\kappa dX_t = 4\bigg(\frac{1}{X^*_t}-\frac{1}{X_t}\bigg)dt+(\kappa^*-\kappa)dW_t.
    \end{equation}
    Integrate both sides, we write
    \begin{equation*}
        \kappa^*X^*_t-\kappa X_t = (\kappa^*-\kappa)\cdot a+4\int_0^t\frac{X_s-X^*_s}{X^*_sX_s}ds+(\kappa^*-\kappa)W_t.
    \end{equation*}
    On the other hand, inspecting the above equation, we see another term $X^*_t-X_t$ appears in the integrand. Based on Lemma $3.4$, we have
    \begin{equation}\begin{aligned}
    \label{3.28}
        \sup\limits_{0\leq s\leq t}\abs{X^*_s-X_s}\leq\frac{2\sqrt{t}}{\kappa}(\kappa^*-\kappa)^{1/2}.
    \end{aligned}\end{equation}
    At this moment, we have obtained an explicit form of $\kappa^*X^*_t-\kappa X_t$, which is contained in the expression of $d\lambda_k(t)-d\lambda^*_k(t)$. Define $M_t\coloneqq\inf\limits_{0\leq s\leq t}X_s$ as the running infimum of the Bessel process $X_t$. The running infimum $M^*_t$ of $X^*_t$ is similarly defined. We further denote by $M_\infty=\lim\limits_{t\to\infty}M_t$ the infimum of $X_t$. And similarly, we denote by $M^*_\infty=\lim\limits_{t\to\infty}M^*_t$ the infimum of $X^*_t$. Indeed, from (\cite{Folland}, $Eqn.\;2.1$) and the Brownian scaling property that $X_t$ being a Bessel process starting from $a$ implies $a^{-1}X_{a^2t}$ being a Bessel process starting from $1$, we know that
    \begin{equation*}\begin{aligned}
        \mathbb{P}(M_\infty<y)&=\frac{y^{2\nu}}{a^{2\nu}}\cdot\mathbbm{1}_{y\in[0,a]},\\
        \mathbb{P}(M^*_\infty<y)&=\frac{y^{2\nu^*}}{a^{2\nu^*}}\cdot\mathbbm{1}_{y\in[0,a]}.
    \end{aligned}\end{equation*}
Combining $(3.20)$, $(3.26)$ and $(3.28)$, we have that
    \begin{equation}\begin{aligned}
    \label{3.30}
        &\abs{\lambda_k(t)-\lambda_k^*(t)}\leq(\kappa^*-\kappa)\frac{2a}{\kappa^*\kappa}\cdot\frac{t}{M^*_tM_t}+(\kappa^*-\kappa)^{\frac{1}{2}}\frac{16}{\kappa^*\kappa^2}\cdot\frac{t^\frac{5}{2}}{(M^*_tM_t)^2}\\
        &\;\;\;\;\;\;\;+(\kappa^*-\kappa)\frac{2}{\kappa^*\kappa}\cdot\frac{t}{M^*_tM_t}\sup\limits_{0\leq s\leq t}\abs{W_s}.
    \end{aligned}\end{equation}
    Considering $\kappa^*>\kappa$, we have that
    \begin{equation*}\begin{aligned}
        &\sup\limits_{t\in[0,T]}\abs{\lambda_k(t)-\lambda_k^*(t)}\leq(\kappa^*-\kappa)\frac{2a}{\kappa^2}\cdot\frac{T}{M^*_\infty M_\infty}+(\kappa^*-\kappa)^\frac{1}{2}\frac{16}{\kappa^3}\cdot\frac{T^\frac{5}{2}}{(M^*_\infty M_\infty)^2}\\
        &\;\;\;\;\;\;\;\;\;\;\;\;\;\;\;\;\;\;\;\;\;\;\;\;\;\;\;+(\kappa^*-\kappa)\frac{2}{\kappa^2}\cdot\frac{T}{M^*_\infty M_\infty}\sup\limits_{0\leq s\leq T}\abs{W_s}.
    \end{aligned}\end{equation*}
    Based on $(3.30)$, define the following events
    \begin{equation}\begin{aligned}
    \label{3.32}
        E_1&\coloneqq\big\{M_\infty\geq(\kappa^*-\kappa)^{\frac{1}{16}}\big\},\\
        E_2&\coloneqq\big\{M^*_\infty\geq(\kappa^*-\kappa)^{\frac{1}{16}}\big\},\\
        E_3&\coloneqq\bigg\{\sup\limits_{0\leq s\leq T}\abs{W_s}\leq\frac{1}{(\kappa^*-\kappa)^{\frac{3}{4}}}\bigg\}.
    \end{aligned}\end{equation}
    Since $\kappa^*>\kappa$ and by $(3.30)$, we know that
    \begin{equation}
    \label{3.33}
        \mathbb{P}\big(E_1\big)=1-\frac{(\kappa^*-\kappa)^\frac{\nu}{8}}{a^{2\nu}}\;\;\;\text{and}\;\;\;\mathbb{P}\big(E_2\big)=1-\frac{(\kappa^*-\kappa)^\frac{\nu^*}{8}}{a^{2\nu^*}}.
    \end{equation}
    From (\cite{Ben}, $Cor.\;2.2$), we know the supremum Brownian motion $S_t$ admits the following distribution 
 \begin{equation*}
     \mathbb{P}\big(S_t\leq x\big)=2\Phi\bigg(\frac{x}{\sqrt{t}}\bigg)-1
 \end{equation*}
 for all $x\geq0$ and where $\frac{d}{dx}\Phi(x)\coloneqq e^{-x^2/2}/\sqrt{2\pi}$ is the density of standard normal variable. It follows from the reflection principle that\begin{equation}
 \label{3.35}
     1-\mathbb{P}\big(E_3\big)=2\mathbb{P}\big(S_T\geq(\kappa^*-\kappa)^{-\frac{3}{4}}\big)\leq2\sqrt{\frac{2}{\pi}}(\kappa^*-\kappa)^\frac{3}{4}\cdot e^{-\frac{1}{2T(\kappa^*-\kappa)^{3/2}}}.
 \end{equation}
 Choose $\alpha_1\coloneqq C(T,G)\frac{4T}{\kappa^2},\;\alpha_2\coloneqq C(T,G)\frac{32T^{\frac{5}{2}}}{\kappa^3},\;\alpha_3\coloneqq C(T,G)\frac{4Ta}{\kappa^2}$. It follows from Proposition $2.2$ and $(3.32)$ that on the event $E_1\cap E_2\cap E_3\subseteq\Omega$, we have the estimate
 \begin{equation*}\begin{aligned}
     &\norm{g(t,z)-g^*(t,z)}_{[0,T]\times G}\leq C(T,G)\sum\limits_{k=1}^2\norm{\lambda_k(t)-\lambda^*_k(t)}_{[0,T]}\\
     &\;\;\;\;\;\;\;\;\;\;\;\;\;\;\;\;\;\leq\alpha_1(\kappa^*-\kappa)^{1/8}+\alpha_2(\kappa^*-\kappa)^{1/4}+\alpha_3(\kappa^*-\kappa)^{7/8}.
 \end{aligned}\end{equation*}
 On the other hand, from $(3.33)$ and $(3.35)$ we have
 \begin{equation*}
     \mathbb{P}\big(E_1\cap E_2\cap E_3\big)\geq1-2\frac{(\kappa^*-\kappa)^\frac{\nu}{8}}{a^{2\nu}}-2(\kappa^*-\kappa)^\frac{3}{4}\cdot e^{-\frac{1}{2T(\kappa^*-\kappa)^{3/2}}}
 \end{equation*}
 where $a=a_1-a_2>0$. Hence, the result is verified.
\end{proof}
\begin{remark}
    Similar to the remark following Proposition $3.1$, the above result deals with random sets $H_T,K_T\subseteq\mathbb{H}$, which depend pathwisely on the forward Loewner chain.
\end{remark}
\begin{corollary}
  Suppose there is a sequence of forward Loewner chains $g^n_t(z): \mathbb{H}\backslash K^n_T\to \mathbb{H}$, with hulls $K_T^n$, induced by Dyson Brownian motions $\{\lambda^n_1(t),\lambda^n_2(t)\}$ with parameter $\kappa_n$ such that $\lim\limits_{n\to\infty}\kappa_n=\kappa \in (0, 4]$. Suppose $g_t(z):\mathbb{H}\backslash K_T\to\mathbb{H}$ is a forward Loewner chain induced by Dyson Brownian motion $\{\lambda_1(t),\lambda_2(t)\}$ with diffusivity parameter $\kappa \in (0,4]$. Then, for all $\epsilon>0$, there exists $N_\epsilon\in\mathbb{N}$ such that with $n>N_\epsilon$, we have the estimate $\mathbb{P}(g^n_t\xRightarrow{Cara}g_t)\geq1-\epsilon$ in the domain $z\in\mathbb{H}\backslash(K_T\cup_1^\infty K_T^n)$.
\end{corollary}

\subsection{Remarks on extension to multiple processes}
In the above two sub-sections we have discussed the stability of multiple $SLE_\kappa$ under perturbations of initial value and diffusivity parameter with $N=2$ Dyson Brownian motions. In this sub-section, we formalize the general $N\geq3$ case and point out further directions and remarks.\par
The general driving forces $\{\lambda_1(t),\lambda_2(t),\ldots,\lambda_N(t)\}$ are interacting Dyson Brownian motions given by
\begin{equation*}
    d\lambda_j(t)=\frac{1}{\sqrt{2}}dB_j(t)+\frac{2}{\kappa}\sum\limits_{1\leq k\leq N,k\neq j}\frac{1}{\lambda_j(t)-\lambda_k(t)}dt
\end{equation*}
with initial value $\lambda_j(0)=a_j$ in the higher dimensional Weyl chamber for all $j=1,\ldots, N$, where $B_1(t),\ldots,B_N(t)$ are independent one-dimensional Brownian motions. Here the trajectory of the forces $\{\lambda_1(t),\ldots,\lambda_N(t)\}$ can be seen as the time evolution of $N$ non-colliding particles with hardcore. Furthermore, we set diffusivity $\kappa\in(0,4]$ and for each $1\leq i\neq j\leq N$ and let $X^{i,j}_t=\lambda_i(t)-\lambda_j(t)$. Based on Section $1$ we know $X^{i,j}>0$ for all $t\in[0,T]$ almost surely whenever $i>j$. In other words, the particles never collide in the simple phase of $\kappa$. We further observe that
\begin{equation}
\label{3.39}
    dX^{i,j}_t=\frac{4}{\kappa}\cdot\frac{dt}{X^{i,j}_t}+\frac{2}{\kappa}\sum\limits_{k\neq i,j}\bigg(\frac{1}{X^{i,k}_t}-\frac{1}{X^{j,k}_t}\bigg)dt+dW^{i,j}_t
\end{equation}
with $X^{i,j}_0=a_i-a_j$ and where $W^{i,j}_t=\frac{1}{\sqrt{2}}\big(B_i(t)-B_j(t)\big)$ is a canonical Wiener process.\par
If the initial value $\{\lambda_1(0),\lambda_2(0),\ldots,\lambda_N(0)\}$ is under perturbation, we get a different set of Dyson Brownian motions. To be precise, choose a positive constant $0<\epsilon<\min\{\frac{1}{3}(a_j-a_i);\,1\leq i<j\leq N\}$ and select $b_j$ in the $\epsilon$-ball of $a_j$ for $j=1,\ldots,N$ to be the perturbed initial value of the Dyson Brownian motions $\{\eta_1(t),\eta_2(t),\ldots,\eta_N(t)\}$. Hence $\{b_1,b_2,\ldots,b_N\}$ stays in the Weyl chamber and 
\begin{equation*}
    d\eta_j(t)=\frac{1}{\sqrt{2}}dB_j(t)+\frac{2}{\kappa}\sum\limits_{1\leq k\leq N,k\neq j}\frac{1}{\eta_j(t)-\eta_k(t)}dt
\end{equation*}
with initial value $\eta_j(0)=b_j$ for all $j$. Let $Y^{i,j}_t=\eta_i(t)-\eta_j(t)$ for all $i\neq j$ and we see that
\begin{equation}
\label{3.41}
    dY^{i,j}_t=\frac{4}{\kappa}\cdot\frac{dt}{Y^{i,j}_t}+\frac{2}{\kappa}\sum\limits_{k\neq i,j}\bigg(\frac{1}{Y^{i,k}_t}-\frac{1}{Y^{j,k}_t}\bigg)dt+dW^{i,j}_t
\end{equation}
with $Y^{i,j}_0=b_i-b_j$. Observe from $(3.39)$ and $(3.41)$ that $X^{i,j}_t$ and $Y^{i,j}_t$ are no longer Bessel processes due to the interaction terms.\par
To understand the behavior of the forward Loewner chain under the perturbation of initial value, we must look into the process $X^{i,j}_t-Y^{i,j}_t$ for each $i\neq j$. Indeed, we see that
\begin{equation*}
    d(X^{i,j}_t-Y^{i,j}_t)=-\frac{4}{\kappa}\cdot\frac{X^{i,j}_t-Y^{i,j}_t}{X^{i,j}_tY^{i,j}_t}dt-\frac{2}{\kappa}\sum\limits_{k\neq i,j}\bigg(\frac{X^{i,k}_t-Y^{i,k}_t}{X^{i,k}_tY^{i,k}_t}-\frac{X^{j,k}_t-Y^{j,k}_t}{X^{j,k}_tY^{j,k}_t}\bigg)dt.
\end{equation*}
If for all $i\neq j$ and $t\in[0,T]$ we have $X^{i,j}_t-Y^{i,j}_t\neq0$, then define 
\begin{equation*}
    Q^{i,j}_t=\sum\limits_{k\neq i,j}\bigg(\frac{1}{X^{i,k}_tY^{i,k}_t}\cdot\frac{X^{i,k}_t-Y^{i,k}_t}{X^{i,j}_t-Y^{i,j}_t}-\frac{1}{X^{j,k}_tY^{j,k}_t}\cdot\frac{X^{j,k}_t-Y^{j,k}_t}{X^{i,j}_t-Y^{i,j}_t}\bigg).
\end{equation*}
Denote by $a_{ij}=a_i-a_j$ and $b_{ij}=b_i-b_j$. Then we observe
\begin{equation*}
    X^{i,j}_t-Y^{i,j}_t=(a_{ij}-b_{ij})\cdot\exp\bigg(-\frac{4}{\kappa}\int_0^t\frac{1}{X^{i,j}_sY^{i,j}_s}ds-\frac{2}{\kappa}\int_0^tQ^{i,j}_sds\bigg).
\end{equation*}
According to Proposition $2,2$, it would be enough to understand the perturbation $\lambda_j(t)-\eta_j(t)$ for each $j$. Since
\begin{equation*}\begin{aligned}
    d\big(\lambda_j(t)-\eta_j(t)\big)&=\frac{2}{\kappa}\sum\limits_{1\leq k\leq N,k\neq j}\bigg(\frac{1}{\lambda_j(t)-\lambda_k(t)}-\frac{1}{\eta_j(t)-\eta_k(t)}\bigg)dt\\
    &=-\frac{2}{\kappa}\sum\limits_{1\leq k\leq N,k\neq j}\frac{X^{j,k}_t-Y^{j,k}_t}{X^{j,k}_tY^{j,k}_t}dt, 
\end{aligned}\end{equation*}
we know that it would be enough to understand the perturbation induced by the process $Q^{i,j}_t$. However, to have a more quantitative picture, it is necessary to investigate the stochastic differential equations $(3.39)$ and $(3.41)$, which we leave in future projects. And the perturbation induced by the diffusivity parameter $\kappa$ with $N\geq3$ Dyson Brownian motions could be analysed similarly, which we also leave in future projects.

\section{Variant estimate on the Hausdorff distance}
In this section, we prove a variant pathwise perturbation estimate on the Hausdorff distance of the hulls $K_T$ induced by the forward Loewner chain. When $\kappa$ is in $(0,4]$, the multiple $SLE_\kappa$ curves are almost surely simple and non-intersecting. The result in this section serves as a motivation to understand the Hausdorff distance convergence following the analysis of the Carathéodory convergence. In general, the topology induced by Hausdorff distance is stronger than the topology induced by Carathéodory convergence.

We analyze the Hausdorff convergence under assumptions on an upper bound of the modulus of the $z$-derivative of the Loewner map. The method follows the one-curve strategy from (\cite{Onthe}, $Lem.\,8.2$). First, we define the notion of Hausdorff distance. For any two compacts $A,B\subseteq\mathbb{C}$, define their Hausdorff distance (\cite{Alex}, $Sec.\;6.1$) by
\begin{equation*}
    d_H(A,B)\coloneqq\inf\bigg\{\epsilon>0;\,A\subseteq\bigcup\limits_{z\in B}\mathcal{B}(z,\epsilon),\,B\subseteq\bigcup\limits_{z\in A}\mathcal{B}(z,\epsilon)\bigg\},
\end{equation*}
where $\mathcal{B}(z,\epsilon)$ is the $\epsilon$-ball centered at $z\in\mathbb{C}$.\par
Going back to the $N$-curve case. First consider a forward Loewner chain $g_t(z):\mathbb{H}\backslash K_T\to\mathbb{H}$ driven by forces $t\mapsto\big(\lambda_1(t),\ldots,\lambda_N(t)\big)$. Denote the inverse map corresponding to $g_t(z)$ by $f_t(z):\mathbb{H}\to\mathbb{H}\backslash K_T$, with $g_t\big(f_t(z)\big)=z$, for all $z\in\mathbb{H}$. On the other hand, consider the backward Loewner chain generated by the time-reversed forces $t\mapsto\big(\lambda_1(T-t),\ldots,\lambda_N(T-t)\big)$. Denote this forward Loewner chain by $h_t(z)$ for $z\in\mathbb{H}$. Then it satisfies
\begin{equation*}
    \partial_th_t(z)=\frac{1}{N}\sum\limits_{j=1}^N\frac{-2}{h_t(z)-\lambda_j(T-t)}
\end{equation*}
with $h_0(z)=z$ for all $z\in\mathbb{H}$. Similar to the $N=1$ case as in (\cite{Alex2}, $Sec.\;2.$), we could verify that $f_T(z)=h_T(z)$ for all $z\in\mathbb{H}$. When the system is under perturbation, we need to compare a Loewner chain with its perturbed counterpart. Indeed, denote by $f_k(t,z)$ and $g_k(t,z)$ the Loewner chains driven by $\{V_{k,1}(t),\ldots,V_{k,N}(t)\}$ for $k=1,2$. Denote by $h_k(t,z)$ the backward Loewner chains driven by $\{V_{k,1}(T-t),\ldots,V_{k,N}(T-t)\}$ for $k=1,2$.
The following lemma estimates pathwisely the backward Loewner chain.
\begin{lemma}
    For all $\delta>0$, there exists a constant $C(\delta,T)=\sqrt{1+4T/\delta^2}$ such that whenever $\Im z\geq\delta$, we have
    \begin{equation*}
        \abs{h_1(T,z)-h_2(T,z)}\leq C(\delta,T)\sum\limits_{j=1}^N\norm{V_{1,j}(T-t)-V_{2,j}(T-t)}_{[0,T]}.
    \end{equation*}
\end{lemma}
\begin{proof}
    The proof is similar to Proposition 2.2. Take $I_{k,j}\leq \log\frac{\Im h_k(t,z)}{\Im z}.$
\end{proof}
We also need the following  Koebe
distortion theorem, see ([14] $Lem.\;2.1$)
\begin{lemma}
    Let $D$ be a simply connected domain and assume $f:D\to\mathbb{C}$ is conformal map. Let $d=\text{dist}(z,\partial D)$ for $z\in D$. If $\abs{z-w}\leq rd$ for some $0<r<1$, then
    \begin{equation*}
        \frac{\abs{f^\prime(z)}}{(1+r)^2}\abs{z-w}\leq\abs{f(z)-f(w)}\leq\frac{\abs{f^\prime(z)}}{(1-r)^2}\abs{z-w}.
    \end{equation*}
\end{lemma}
\begin{proposition}
    Let $g_k(t,k):[0,T]\times\mathbb{H}\backslash K_{k,t}$ be two forward Loewner chains driven by forces $t\mapsto\big(V_{k,1}(t),\ldots,V_{k,N}(t)\big)$ with hulls $K_{k,t}$, for $k=1,2$. Let $f_k(t,z)$ be their inverse so that $g_k\big(t,f_k(t,z)\big)=z$. Write $f_k(z)\coloneqq f_k(T,z)$, for $k=1,2$. Suppose that 
    \begin{equation*}
        \sum\limits_{j=1}^N\sup\limits_{0\leq t\leq T}\abs{V_{1,j}(t)-V_{2,j}(t)}<\epsilon,
    \end{equation*}
    where $\epsilon>0$ is taken sufficiently small. Suppose further there exists $\theta \in(0,1)$ such that for all $\zeta\in\mathbb{R}$, we have 
    \begin{equation}
    \label{4.6}
        \abs{f_1^\prime(\zeta+i\delta)}\leq\delta^{-\theta}
    \end{equation}
    for all $\delta\leq4\sqrt{T\epsilon}$. Then, we have the Hausdorff metric estimate
    \begin{equation*}
        d_H(K_{1,T}\cup\mathbb{R},K_{2,T}\cup\mathbb{R})\leq8(T\epsilon)^\frac{1-\theta}{2}+3\sqrt{\epsilon(1+\epsilon)}.
    \end{equation*}
\end{proposition}
\begin{proof}
    Denote by $h_k(t,z)$ the time-reversed Loewner chains driven by $\{V_{k,1}(T-t),\ldots,V_{k,N}(T-t)\}$ for $k=1,2$. Based on Lemma $4.1$ and the observation that $f_k(z)=h_k(T,z)$, we know
    \begin{equation*}
        \abs{f_1(z)-f_2(z)}\leq\epsilon\cdot\sqrt{1+4T/\delta^2}
    \end{equation*}
    whenever $\Im z\geq\delta$. Take $\delta_0=4\sqrt{T\epsilon}$, we have
    \begin{equation*}
        \sup\limits_{\Im z\geq\frac{\delta_0}{2}}\abs{f_1(z)-f_2(z)}\leq\sqrt{\epsilon(1+\epsilon)}.
    \end{equation*}
    Hence, Cauchy's integral formula implies
    \begin{equation*}\begin{aligned}
        \sup\limits_{\Im z\geq\delta_0}\abs{f_1^\prime(z)-f_2^\prime(z)}&\leq\sqrt{\epsilon(1+\epsilon)}\sup\limits_{\Im z\geq\frac{\delta_0}{2}}\frac{1}{2\pi i}\oint_{\partial\mathcal{B}(z,\delta_0/2)}\frac{d\zeta}{\abs{z-\zeta}^2}\\
        &\leq\sqrt{\frac{1+\epsilon}{4T}}.
    \end{aligned}\end{equation*}
    For notational convenience, we write $\widehat{K}_k\coloneqq K_{k,T}\cup\mathbb{R}$, for $k=1,2$. Fix $\zeta\in\mathbb{R}$, by Lemma $4.2$, we have
    \begin{equation*}
        \abs{f_1(\zeta+i0^+)-f_1(\zeta+i\delta)}\leq\delta\cdot\abs{f_1^\prime(\zeta+i\delta)}\leq\delta^{1-\theta}\leq(16T\epsilon)^\frac{1-\theta}{2}.
    \end{equation*}
    Hence, we have
    \begin{equation*}
        f_1\big(\{\Im z\leq\delta_0\}\big)\subseteq\bigcup\limits_{z\in\widehat{K}_1}\mathcal{B}\big(z,(16T\epsilon)^\frac{1-\theta}{2}\big).
    \end{equation*}
    We have that 
    \begin{equation*}
        \widehat{K}_2\subseteq f_2\big(\{\Im z\leq\delta_0\}\big).
    \end{equation*}
    For the above fixed $\zeta\in\mathbb{R}$, write $w\coloneqq f_1(\zeta+i0^+)\in\widehat{K}_1$. Choose $\widehat{w}\in\widehat{K}_2$ be the point in $\widehat{K}_2$ nearest to $f_2(\zeta+i\delta_0)$.
    
    
    By Lemma $4.2$ again
    \begin{equation*}\begin{aligned}
        \abs{\widehat{w}-f_2(\zeta+i\delta_0)}&\leq\abs{f_2(\zeta+i0^+)-f_2(\zeta+i\delta_0)}\\
        &\leq\delta_0\cdot\abs{f_2^\prime(\zeta+i\delta_0)}\\
        &\leq\delta_0\cdot\bigg(\abs{f_1^\prime(\zeta+i\delta_0)}+\abs{f_1^\prime(\zeta+i\delta_0)-f_2^\prime(\zeta+i\delta_0)}\bigg)\\
        &\leq(16T\epsilon)^{\frac{1-\theta}{2}}+\sqrt{4\epsilon(1+\epsilon)}.
    \end{aligned}\end{equation*}
    Hence, we see that
    \begin{equation*}\begin{aligned}
        \abs{w-\widehat{w}}&\leq\abs{w-f_1(\zeta+i\delta_0)}+\abs{f_1(\zeta+i\delta_0)-f_2(\zeta+i\delta_0)}+\abs{\widehat{w}-f_2(\zeta+i\delta_0)}\\
        &\leq(16T\epsilon)^{\frac{1-\theta}{2}}+\sqrt{\epsilon(1+\epsilon)}+(16T\epsilon)^{\frac{1+\theta}{2}}+\sqrt{4\epsilon(1+\epsilon)}\\
        &\leq8(T\epsilon)^\frac{1-\theta}{2}+3\sqrt{\epsilon(1+\epsilon)}.
    \end{aligned}\end{equation*}
    Hence the result is verified.
\end{proof}
\begin{remark}
    The upper bound for the derivative modulus in $(\ref{4.6})$ may seem superfluous. But we could estimate the probability of which $(\ref{4.6})$ holds using a multi-dimensional variant argument as in (\cite{ooo}, $Cor.\;3.5$)
\end{remark}

\textbf{Acknowledgements.} We kindly thank Henri Elad Altman for his valuable comments on analyzing stochastic processes and to Vivian Olsiewski Healey and Dmitry Beliaev for useful discussions and for looking over the previous versions of this manuscript. We also kindly thank Andrew Swan, and Lukas Schoug for helpful discussions. V.M. acknowledges the support of the NYU-ECNU Institute of Mathematical Sciences at NYU Shanghai.

\bibliographystyle{plain}
\bibliography{literature}
\begin{spacing}{1}

\end{spacing}
\end{document}